\documentclass[11pt]{article}
\usepackage[english]{babel}
\usepackage{hyperref}
\usepackage{amsmath,amssymb,amstext}
\usepackage{amsthm}
\usepackage{dsfont}

\newcommand{\E}{\mathds{E}}

\newcommand{\N}{\mathds{N}}
\newcommand{\1}{\mathds{1}}

\newcommand{\vol}{\operatorname{vol}}
\newcommand{\conv}{\operatorname{conv}}
\newcommand{\bd}{\operatorname{bd}}
\renewcommand{\phi}{\varphi}

\newtheorem{lemma}{Lemma}

\newtheorem{theorem}{Theorem}

\begin{document} 
\begin{center} \LARGE On the monotonicity of the moments of volumes of random simplices \normalfont\normalsize

\vspace{1cm}

Benjamin Reichenwallner and Matthias Reitzner

\vspace{0.1cm}

\textit{University of Salzburg and University of Osnabrueck}

\vspace{0.2cm}

\rule{5cm}{0.01cm}

\end{center}

\abstract{In a $d$-dimensional convex body $K$ random points $X_0, \dots, X_d$ are chosen. Their convex hull is a random simplex. The expected volume of a random simplex is monotone under set inclusion, if $K \subset L$ implies that the expected volume of a random simplex in $K$ is smaller than  the expected volume of a random simplex in $L$. Continuing work of Rademacher, it is shown that moments of the volume of random simplices are in general not monotone under set inclusion.}

\bigskip

\section{Introduction}

For a $d$-dimensional convex body $K$, denote the volume of the convex hull of $d+1$ independently uniformly distributed random points $X_0, \ldots, X_d$ in $K$ by $V_K$. Since the $d+1$ points are in general position with probability $1$, their convex hull is almost surely a full-dimensional simplex. Meckes~\cite{meck} asked whether the expected volume $\E V_K$ is monotone under inclusion, i.e. for each pair of convex bodies $K, L \subseteq \mathbb{R}^d$, $K \subseteq L$ implies 
\begin{equation}\label{conj:meckes-strong}
\E V_K \leq \E V_L. 
\end{equation}
He also states a weak conjecture concerning the existence of a universal constant $c>0$ such that $K \subseteq L$ implies 
$$\E V_K \leq c^d\, \E V_L.$$
Interest in this question comes from the fact that both conjectures would imply a positive solution to the slicing problem. In fact, the weak conjecture is an equivalent formulation, see e.g. \cite{rad}. For a more general statement of the conjecture we refer to \cite{reitz_surv}.

In this paper, we investigate the question for arbitrary moments of the volumes of random simplices. 
Let $K, L \subseteq \mathbb{R}^d$ denote convex bodies. 
For $d \in \N$, define $k_d$ as the critical exponent such that 
	\begin{enumerate}
		\item[(i)] $K \subseteq L$ implies $\E V_K^k \leq \E V_L^k $ for all $k < k_d$, and
		\item[(ii)] there exist $K \subseteq L$ with $\E V_K^k > \E V_L^k  $ for all $k \geq k_d$. 
	\end{enumerate}
	
At first it is unclear whether there is a critical exponent at all, where the behaviour switches from monotonicity to the existence of counterexamples precisely at  $k_d$. This issue will be a byproduct of the following results. And the interesting second question is whether $k_d < \infty$.

In 2012, Rademacher~\cite{rad} showed that Meckes' stronger conjecture \eqref{conj:meckes-strong} is not true in general.  (Note that the problem is trivial in dimension $1$ where $k_1=\infty$ and monotonicity holds.) More precisely, Rademacher proved the following, where the main point is the surprising existence of counterexamples, i.e. $k_d< \infty$: 
	\begin{theorem}\label{main-Rademacher}
In the planar case, $3 \leq k_2  < \infty$, and in dimension three, $1 \leq k_3 < \infty$ holds. In higher dimension, $k_d=1$ holds for all $d \geq 4$.
	\end{theorem}

Our main theorem computes the constant $k_2$ and makes $k_3$ more precise.

\begin{theorem}\label{main}
In the planar case, $k_2 = 3$ holds. In dimension three, $k_3 \in\{ 1,2 \}$. 
\end{theorem}

The reader might recognize that there is still one open task, namely to prove $k_3=1$, i.e. to disprove monotonicity of the expected volume of a random tetrahedron in dimension three. Numerical simulations show that there is a counterexample (as already conjectered by Rademacher), but a rigorous proof is still missing. 

Since a direct proof of this issue is somewhat involved, one may use a very crucial lemma, stated here. In the following, we denote by $V_{K,x}$ the volume of a random simplex, which is the convex hull of a fixed point $x$ and $d$ independent uniform random points in $K$.

\begin{lemma}[Rademacher~\cite{rad}]\label{equ}
	For $k,d \in \N$, monotonicity under inclusion of the map 
	$$K \mapsto \E V_K^k,$$ 
	where $K$ ranges over all $d$-dimensional convex bodies, holds if and only if we have for each convex body $K \subseteq \mathbb{R}^d$ and for each $x \in \bd K$ that 
	$$\E V_K^k \leq \E V_{K,x}^k.$$
\end{lemma}

The lemma allows us to consider one convex body $K$, rather than a pair of convex bodies, and compute two different moments: the moment of the volume of a random simplex in $K$ as well as the same, but fixing one of the $d+1$ points to be a point on the boundary of $K$, denoted by $\bd K$. 

Rademacher takes $K$ to be a $d$-dimensional halfball and $x$ is the midpoint of the base which is a $(d-1)$-dimensional ball. 
These form the counterexamples to the monotonicity in Theorem~\ref{main-Rademacher} for all moments for $d \geq 4$, and to the monotonicity of all but finitely many moments for $d=2,3$. 

In the background of our Theorem~\ref{main}, there are a more detailed computation of Rademacher's counterexample in dimension $3$ and the construction of a new counterexample in dimension $2$. Here we have to compute the area of a random triangle in a triangle where one vertex is fixed at the midpoint of one edge. 
\begin{theorem}\label{triangle}
	Let $T \subseteq \mathbb{R}^2$ be a triangle and $x$ the midpoint of an edge of $T$. Then we have for the $k$-th moment of $V_{T,x}=
	\vol \conv (x, X_1,X_2)$:
	$$\frac{\E V_{T,x}^k}{\vol T^k} = \frac{2^{3-k}}{(k+1)(k+2)^2(k+3)} \left(\sum\limits_{l=1}^{k+1} \binom{k+2}{l}^{-1} + 1\right).$$
\end{theorem}

Coming back to the (lack of) monotonicity of the expected volume of a random tetrahedron in dimension $3$, it has already been conjectured by Rademacher that the above example, the halfball $B_3^+$ together with one point at the origin $o$ of its base, should also form a counterexample in this case. The value $\E V_{B_3^+,o}$  is known (and will be given in Section \ref{prel}), but for $\E V_{B_3^+}$, the precise value is an open task. Numerical computations show that
$$ 
0.028105 \approx \E V_{B_3^+}  > \E V_{B_3^+,o} = \frac{9\pi}{ 1\,024} = 0.0276\ldots . $$
A second counterexamle is given by a tetrahedron $T$ and $x$ the centroid of one of its facets. Here $\E V_T$ is known, but $\E V_{T,x}$ is missing. Again, by numerical integration we obtain
$$
0.0173\ldots  = \frac{13}{720}- \frac{\pi^2}{15\,015} =  \E V_T > \E V_{T,x} \approx 0.015901 . $$

This paper is organized in the following way. In Section~\ref{prel}, we give some auxiliary results and notation. Then we compute the moments of the area of a random triangle inside a triangle. This result will be used in Section~\ref{sec:main} for the proof of the main theorem, which is an extension of two theorems of Rademacher. 

As a general reference for the tools and results we need in the following, we refer to the book on Stochastic and Integral Geometry by Schneider and Weil \cite{schWeil}. More recent surveys on random polytopes are due to Hug \cite{hug_surv} and Reitzner \cite{reitz_surv}.

\section{Preliminaries}\label{prel}

In the following, we need some well-known results on random polytopes which we collect here for later use. We start with dimension one, where a convex set is an intervall $I$ and the volume $V_I$ of a random simplex is the distance between two random points.
\begin{lemma}[cf., e.g., \cite{sol}]\label{solomon}
	Assume $I$ is an intervall of length $l$. Then 
	$$\E V_I^k = \frac{2 l^k}{(k+1)(k+2)}.$$
\end{lemma}

The only convex body for which the moments of random simplices are known in all dimensions is the unit ball $B_d$. They were computed by Miles \cite{mil}.
Let $\kappa_d = \vol B_d = \pi^{d/2}/\Gamma(1+d/2)$ be the volume of $B_d$, where $\Gamma(\cdot)$ denotes the gamma function, and let $\omega_d = \vol S_{d-1} = d \kappa_d$ be the $(d-1)$-dimensional volume of the boundary of $B_d$.

\begin{theorem}[cf. \cite{mil} or \cite{schWeil}, Theorem 8.2.3]\label{ball}
	For any $d,k\in \N$, we have
	$$\E V_{B_d}^k = \frac{1}{(d!)^k} \left(\frac{\kappa_{d+k}}{\kappa_d}\right)^{d+1} \frac{\kappa_{d(d+k+1)}}{\kappa_{(d+1)(d+k)}} \frac{\omega_1 \cdots \omega_k}{\omega_{d+1} \cdots \omega_{d+k}}.$$
\end{theorem}
Observe that these values coincide with those from Lemma \ref{solomon} for $d=1$ and $l=2$.
It was proved by Blaschke (for $d=3$) and Groemer (for arbitrary $d$) that these values are extremal in the sense that under all convex bodies of volume one, $\E V_K^k$ is minimized for the ball.

\begin{theorem}[Blaschke-Groemer, cf.~\cite{schWeil}, Theorem 8.6.3]\label{bla_groe}
	Let $d,k\in \N$. Among all $d$-dimensional convex bodies, the map
	$$K \mapsto \frac{\E V_K^k}{\vol K^k}$$ 
	attains its minimum  if and only if $K$ is an ellipsoid.
\end{theorem}

In the following, we need the expected volume of a random simplex where one point is fixed at the origin and the others are uniformly chosen in the unit ball. Again, this result is due to Miles.

\begin{theorem}[cf. \cite{mil} or \cite{schWeil}, Theorem 8.2.2]\label{ball_origin}
	For any $d, k\in \N$, we have
	$$\E V_{B_d,o} ^k = \frac{1}{(d!)^k} \left(\frac{\kappa_{d+k}}{\kappa_d}\right)^{d} \frac{\omega_1 \cdots \omega_k}{\omega_{d+1} \cdots \omega_{d+k}}.$$
\end{theorem}

Finally, we prove here a general lemma that has already been used by Rademacher~\cite{rad} in the case of the unit ball. It seems to be well-known, but we could not find a rigorous proof in the literature.  

\begin{lemma}\label{symm}
	Assume $K$ is a $d$-dimensional convex body which is symmetric with respect to $x \in K$ and let $H^+$ be a half-space containing $x$ on its boundary. Then 
	$$\E V_{K \cap H^+,x}^k = \E V_{K,x}^k.$$
\end{lemma}

\begin{proof}
	Without loss of generality, we identify $x$ with the origin $o$. Furthermore, we denote the intersection of $K$ with the half-space $H^+$ by $K^+$. Since the volume of the simplex $\conv(o,x_1,\ldots, x_d)$ is just the absolute value of the determinant of the matrix containing the vectors $x_i$, divided by $d!$, it holds:
	\begin{align*}
	&\E V_{K^+,o}^k = \frac{1}{d!}\,\E_{X_i\in K^+} |\det(X_1,\ldots, X_d)|^k.
	\end{align*}
	
Assume that $\epsilon_i \in \{\pm 1\}$ for $i=1, \dots d$. Because the absolute value of the determinant is an even function, 
$|\det(X_1,\ldots, X_d)|=|\det(\epsilon_1 X_1,\ldots, \epsilon_d X_d)|$ for any choice of $\epsilon_i$. It follows, summing over all possible combinations of signs,	
	\begin{align*}
	\E V_{K^+,o} ^k 
	= &\frac{1}{2^d\, d!}\,\sum_{\epsilon_i=\pm 1} \E_{X_i\in K^+} |\det(\epsilon_1 X_1,\ldots, \epsilon_d X_d)|^k \\
	= & \frac{1}{2^d\, d!} \frac{1}{(\vol K^+)^d}\,\sum_{\epsilon_i=\pm 1} \int\limits_{(K^+)^d} |\det(\epsilon_1 x_1,\ldots, \epsilon_d x_d)|^k\; d(x_1,\ldots,x_d). 
	\end{align*}
	Since the reflection of each point in $K^+$ lies in $K\setminus K^+$, in fact we integrate over all $d$-tuples of points lying in $K$. Because, due to symmetry, the volume of $K^+$ is just half of the volume of $K$, we get
	\begin{align*}
	\E V_{K^+,o}^k 
	= & \frac{1}{2^d\, d!} \frac{1}{(\vol K^+)^d}\, \int\limits_{K^d} |\det(x_1,\ldots,x_d)|^k\; d(x_1,\ldots,x_d)  \\
	= &\frac{1}{d!}  \,\E_{X_i\in K} |\det(X_1,\ldots,X_d)|^k  
	=  \E V_{K,o}^k.
	\end{align*}	
\end{proof}

Let $H^+$ be any halfspace containing the origin in its boundary. Denote by $B_d^+= B_d \cap H^+$ half of the $d$-dimensional unit ball. Using the lemma above, we immediately see that for any $d, k\in \N$ we have
	\begin{equation}\label{eq:halfball}
\E V_{B_d^+,o}^k = \frac{1}{(d!)^k} \left(\frac{\kappa_{d+k}}{\kappa_d}\right)^{d} \frac{\omega_1 \cdots \omega_k}{\omega_{d+1} \cdots \omega_{d+k}}.	 
	\end{equation}
For evaluation of the occuring expressions --- and in particular of the volume of the unit ball ---, the following estimates are useful.
\begin{lemma}[Borgwardt~\cite{bor}, also cf. \cite{rad}]\label{borg}
	For $d \geq 2$, we have 
	$$\sqrt{\frac{d}{2\pi}} \leq \frac{\kappa_{d-1}}{\kappa_d} \leq \sqrt{\frac{d+1}{2\pi}}.$$
\end{lemma}

\section{Random triangles in a triangle}
Essential for our investigations in the planar case is the expected area of a random triangle in a given triangle $T$. In particular, we need the expected area in the case where one point is fixed at the midpoint of an edge. This is the statement of Theorem \ref{triangle}, which is proved in this section.

Let $x$ be the midpoint of an edge of $T$. We show that 
	$$\frac{\E V_{T,x}^k}{\vol T^k} = \frac{2^{3-k}}{(k+1)(k+2)^2(k+3)} \left(\sum\limits_{l=1}^{k+1} \binom{k+2}{l}^{-1} + 1\right).$$

\begin{proof}
	Since the moments of the volume of the random triangle do not depend on the shape of the triangle, we can consider the specific triangle 
	$$T = \{(x,y)\in \mathbb{R}^2: x,y \geq 0, x+y\leq 1\},$$ 
	i.e. the triangle with vertices $E_0=(0,0), E_1=(1,0)$ and $E_2=(0,1)$. Note that its area is $\vol T =1/2$. We choose $x=(1/2,1/2)$, the midpoint of the edge $\{(x,y)\in\mathbb{R}^2: x,y\geq 0, x+y=1\}$, to be the fixed vertex of the random triangle. 
	
	Using the affine Blaschke-Petkantschin formula --- see e.g. \cite{schWeil} ---, we transform our integral and integrate over all lines $H$ intersecting the triangle.	
	\begin{align*}
	\E V_{T,x}^k & = \frac{1}{\vol T^2} \int\limits_{T^2} \vol\conv(x,x_1,x_2)^k\; d(x_1,x_2)  \\
	&= 4\int\limits_{A(2,1)} \int\limits_{(H\cap T)^2} \vol\conv(x,x_1,x_2)^k\ ||x_1-x_2||\; d(x_1,x_2)\; dH,
	\end{align*}
	where $A(2,1)$ denotes the affine Grassmannian of lines in $\mathbb{R}^2$. We represent a line $H$ by its unit normal vector $u \in S_1$ and its distance $t>0$ from the origin and we therefore denote the line by 
	$$H_{t,u} = \{x \in \mathbb{R}^2: \langle x,u\rangle = t\}.$$ 
	We choose the normalization of the Haar measure $dH$ in such a way that $dH = dt\; du,$
	where $dt$ and $du$ correspond to Lebesgue measures in $\mathbb{R}$ and $S_1$. The area of the triangle $\conv(x,x_1,x_2)$ is the product of the length $||x_1-x_2||$ of its base and its height $d(H_{t,u},x)$, divided by $2$. We write the appearing integral as an expectation to get 
	\begin{align*}
	\E V_{T,x}^k &= 4 \int\limits_{S_1} \int\limits_0^{\infty} \frac{d(H_{t,u},x)^k}{2^k} \int\limits_{(H_{t,u}\cap T)^2} ||x_1-x_2||^{k+1}\; d(x_1,x_2)\; dt\; du\\
	&= 4 \int\limits_{S_1} \int\limits_0^{\infty} \frac{d(H_{t,u},x)^k}{2^k} \vol(H_{t,u}\cap T)^2\; \E V_{H_{t,u}\cap T}^{k+1}\  dt\; du.
	\end{align*}
The $(k+1)$-st moment of the distance of two random points in the intersection $H_{t,u}\cap T$ has already been given in  Lemma~\ref{solomon}.
Hence, we obtain
	\begin{align*}
	\E V_{T,x}^k &= \frac{2^{3-k}}{(k+2)(k+3)} \int\limits_{S_1} \int\limits_0^{\infty} d(H_{t,u},x)^k \vol(H_{t,u}\cap T)^{k+3}\; dt\; du.
	\end{align*}
	A line $H_{t,u}$ that intersects the triangle $T$ a.s. meets exactly two edges of $T$. It splits $T$ into a triangle and a quadrangle. We say that $H_{t,u}$ cuts off the vertex $E_i$ from $T$ if $E_i$ is contained in the triangular part. Furthermore, we write 
	$$\mathcal I^{(i)} = \int\limits_{S_1} \int\limits_0^{\infty} \1(H_{t,u} \text{ cuts off } E_i \text{ from } T)\; d(H_{t,u},x)^k \vol(H_{t,u}\cap T)^{k+3}\; dt\; du$$
	for $i = 0,1,2$, which gives 
	$$ \E V_{T,x}^k  =  \frac{2^{3-k}}{(k+2)(k+3)} \Big( \mathcal I^{(0)}  + \mathcal I^{(1)}  + \mathcal I^{(2)}   \Big). $$
	We state the following lemma which will be proved right after the end of the proof of this proposition:
	
	\begin{lemma}\label{help}
		It holds:
		\begin{enumerate}
			\item[(i)] $\mathcal I^{(0)} = \frac{1}{2^k (k+1)(k+2)} \sum\limits_{l=1}^{k+1} \binom{k+2}{l}^{-1},$
			\item[(ii)] $\mathcal I^{(1)} = \mathcal I^{(2)} = \frac{1}{2^{k+1} (k+1)(k+2)}.$
		\end{enumerate}
	\end{lemma}
	\noindent
	Utilizing this lemma, we get
	$$
	\E V_{T,x}^k =  \frac{2^{3-2k}}{(k+1)(k+2)^2(k+3)} \left(\sum\limits_{l=1}^{k+1} \binom{k+2}{l}^{-1}+1\right).
	$$
\end{proof}
	
	\noindent \textit{Proof of Lemma~\ref{help}.}
		We start with the computation of $\mathcal I^{(0)}$. We subsitute $z = u/t$ and get with $H_{z} = \{x \in \mathbb{R}^2: \langle x,z\rangle = 1\}$ and $dt\; du = |z|^{-3}\; dz$ that
		\begin{align*}
		\mathcal I^{(0)} &= \int\limits_{\mathbb{R}^2} \1(H_z \text{ cuts off } E_0 \text{ from } T)\; d(H_z,x)^k \vol(H_{z}\cap T)^{k+3} |z|^{-3}\; dz.
		\end{align*}
	With a second substitution by $a = z_1^{-1}, b = z_2^{-1}$, we get by $a$ the abscissa of the point of intersection of a line with the $x$-axis and by $b$ the ordinate of the intersection of the line with the $y$-axis. We write $H_{a,b}$ for the line represented by the parameters $a$ and $b$ and have $H_{a,b} = \{x \in \mathbb{R}^2: \langle x,(a^{-1}, b^{-1})\rangle = 1\}$ and $dz = da\; db/(a^2 b^2)$. Considering the appearing indicator function, we see that a line $H_{a,b}$ cuts off $E_0$ from $T$ --- or, in other words, intersects both catheti of $T$ --- if and only if $a$ and $b$ both lie between $0$ and $1$. Our integral consequently transforms into 	
		\begin{align*}
		\mathcal I^{(0)} &= \int\limits_0^1 \int\limits_0^1 \frac{ d(H_{a,b},x)^k \vol(H_{a,b}\cap T)^{k+3}}{(a^2 + b^2)^{3/2}}\  ab \; da\; db.
		\end{align*}

		We use the notation $l(a,b)=\vol (H_{a,b}\cap T)$ for the length of the intersection of $T$ with a line $H_{a,b}$, and $h(a,b)=d(H_{a,b},x)$ for the distance of this line from $x$. We get by easy computations that
		$$l(a,b) = \sqrt{a^2 + b^2} \ \text{ and } \ h(a,b) = (a+b-2ab)/(2\sqrt{a^2+b^2}).$$ 		
		which yields
		\begin{align*}
		\mathcal I^{(0)} &= \int\limits_0^{1} \int\limits_0^{1}  \frac{l(a,b)^{k+3} h(a,b)^k}{(a^2 + b^2)^{3/2}}\  ab \; da\; db 
		= \frac{1}{2^k} \int\limits_0^{1} \int\limits_0^{1} (a+b-2ab)^k ab \; da\; db\\
		&= \frac{1}{2^k} \int\limits_0^{1} \int\limits_0^{1} ((a(1-b) +b (1-a)))^k ab \; da\; db\\
		&= \frac{1}{2^{k}}\sum_{l=0}^k \binom k{l} \int\limits_0^{1} \int\limits_0^{1} a^{l+1} (1-a)^{k-l} b^{k-l+1} (1-b)^l \; da\; db\\
		&= \frac{1}{2^{k}}\sum_{l=0}^k \binom k{l}  \frac{(l+1)!(k-l)!}{(k+2)!} \frac{l!(k-l+1)!}{(k+2)!} \\
		&= \frac{1}{2^{k} (k+1)(k+2)} \sum_{l=1}^{k+1} {\binom {k+2}l }^{-1}, 
		\end{align*}
		 and we arrive at the expression stated in (i).
	
		Considering statement (ii), we first note that $\mathcal I^{(1)} = \mathcal I^{(2)}$ due to symmetry. Hence it suffices to compute $\mathcal I^{(1)}$. Furthermore, the integrals are affine invariant and we can transform the triangle into a similar one, bringing the point $x$ to the midpoint $(1/2,0)$ of the edge on the $x$-axis, and exchanging the vertices $E_i$ clockwise. Now, a line cutting off $E_1$ from $T$ intersects the triangle in both catheti and $a$ and $b$ lie between $0$ and $1$.  
		
		The function $h(a,b)$ changes its sign at $a=1/2$. As before, we get again $l(a,b) = \sqrt{a^2 + b^2}$, and by another straightforward computation, 
		$$h(a,b) = 
		\begin{cases}
		\frac{b-2ab}{2\sqrt{a^2+b^2}} &\text{ for } 0 \leq a \leq \frac{1}{2},\\
		\frac{2ab-b}{2\sqrt{a^2+b^2}} &\text{ for } \frac{1}{2} \leq a \leq 1.
		\end{cases}$$
		The occuring double integral can be solved by partial integration.	
		\begin{align*}
		&\int\limits_0^1 \int\limits_0^{1/2} \frac{l(a,b)^{k+3} h(a,b)^k}{(a^2 + b^2)^{3/2}}\ ab\; da\; db = \frac{1}{2^k} \int\limits_{0}^{1} \int\limits_0^{1/2} (b-2ab)^k ab \; da\; db \\
		&\hspace{2cm}=  \frac{1}{2^{k+2} (k+1)(k+2)} \int\limits_{0}^{1} b^{k+1}\; db = \frac{1}{2^{k+2} (k+1)(k+2)^2}, 
		\end{align*}	
		and analogously	
		\begin{align*}
		&\int\limits_0^1 \int\limits_{1/2}^1 \frac{l(a,b)^{k+3} h(a,b)^k}{(a^2 + b^2)^{3/2}/(ab)}\; da\; db = \frac{2k+3}{2^{k+2} (k+1)(k+2)^2}.
		\end{align*} 
		A combination of both results yields the proof of statement (2).\qed

\section{Proof of Theorem 2}\label{sec:main}

For the first step of the proof, we refine Rademacher's method of proof of Theorem~\ref{main-Rademacher}.

\subsection{Random polytopes in hemispheres}
Let $B_d^+$ be half of the $d$-dimensional unit ball and $L$ the ball of volume $\vol B_d^+$. According to Theorem~\ref{bla_groe}, 
$$\E V_{B_d^+}^k > \E V_L^k,$$ 
and since $\E V_L^k= 2^{-k} \E V_{B_d}^k$, Theorem~\ref{ball} implies  
\begin{align*}
\E V_{B_d^+}^k &> \frac{1}{2^k (d!)^k} \left(\frac{\kappa_{d+k}}{\kappa_d}\right)^{d+1} \frac{\kappa_{d(d+k+1)}}{\kappa_{(d+1)(d+k)}} \frac{\omega_1 \cdots \omega_k}{\omega_{d+1} \cdots \omega_{d+k}}.
\end{align*}
On the other hand, by equation \eqref{eq:halfball},
$$
\E V_{B_d^+,o}^k 
= \frac{1}{(d!)^k} \left(\frac{\kappa_{d+k}}{\kappa_d}\right)^{d} \frac{\omega_1 \cdots \omega_k}{\omega_{d+1} \cdots \omega_{d+k}}.
$$
Combining these statements, we get 
\begin{align}\label{eq:ratio}
\frac{\E V_{B_d^+,o}^k }{\E V_{B_d^+}^k} < 2^k \frac{\kappa_d}{\kappa_{d+k}}\, \frac{\kappa_{(d+1)(d+k)}}{\kappa_{d(d+k+1)}}.
\end{align}
Lemma~\ref{borg} and the inequality $a/b \leq (a+1)/(b+1)$ for $0\leq a \leq b$ yield 
$$
\frac{\E V_{B_d^+,o}^k }{\E V_{B_d^+}^k} < 2^k \left(\frac{(d+2)\cdots (d+k+1)}{(d(d+k+1)+1)\cdots (d(d+k+1)+k)}\right)^{\frac 12} = q(d,k)^{\frac 12}, 
$$
which is Equation (5) in \cite{rad}.
Now consider the series 
\begin{align*}
q(2,k) = 4^k \ \frac{4\cdots (k+3)}{(2(k+3)+1)\cdots (2(k+3)+k)},\quad k \in \N,
\end{align*}
which is strictly decreasing for $k \geq 4$. This can be shown by a computation of the ratio 
\begin{align*}
\frac{q(2,k+1)}{q(2,k)} =
\frac{ 4 (k+4) (2k+7)(2k+8)}{(3k+7)(3k+8)(3k+9)}   \, .
\end{align*}
Furthermore, $q(2,k)$ is smaller than $1$ for $k=11$, and therefore the same is true for $k \geq 11$. 
The values $k=3, \dots, 10$ remain open and will be discussed in the next subsection.

In order to solve the question in dimension $d=3$, we investigate 
\begin{align*}
\frac{q(3,k+1)}{q(3,k)} &=
 \frac{ 4 (k+5)(3k+13)\cdots (3k+15)}{(4k+13)\cdots (4k+16)}. 
\end{align*}
Again, it can be shown easily that this series is strictly decreasing for $k \geq 2$, and $q(3,k)$ is smaller than $1$ for $k=4$ and therefore also for $k \geq 4$. For $k=2,3$ we directly investigate the fraction 
$$2^k \frac{\kappa_{3}}{\kappa_{k+3}}\frac{\kappa_{4(k+3)}}{\kappa_{3(k+4)}}$$ 
in formula \ref{eq:ratio} and obtain that this equals $1$ for $k=2$ and gives for $k=3$:
$$
2^3 \frac{\kappa_{3}}{\kappa_{6}}\frac{\kappa_{24}}{\kappa_{21}}
= 0,384\ldots <1, 
$$ 
yielding the result. 
\qed

\subsection{Random polygons in a triangle}

It remains to give counterexamples for $d=2$ in the cases  $k=3, \dots, 10$. Here we prove that random triangles in triangles are suitable for our purposes by giving the explicit values.

According to Reed~\cite{reed} und Alagar~\cite{ala}, it holds for a triangle $T$ of volume one that 
\begin{align*}
&\E V_T^k = \frac{12}{(k+1)^3 (k+2)^3 (k+3)(2k+5)}\times\\
&\hspace{3cm}\times \left(6(k+1)^2 + (k+2)^2 \sum_{i=0}^k \binom{k}{i}^{-2}\right).
\end{align*}
Using Proposition~\ref{triangle} and recalling that $x$ is the midpoint of an edge of $T$, we have 
$$
\frac{\E V_{T,x}^k}{\E V_T^k} 
= \frac{(k+1)^2 (k+2)(2k+5)}{3 \cdot 2^{k-1}} \frac{\sum_{l=1}^{k+1} \binom{k+2}{l}^{-1} + 1}{6(k+1)^2 + (k+2)^2 \sum_{i=0}^k \binom{k}{i}^{-2}}.
$$
We evaluate this expression for $k = 3,\ldots,10$. 
\begin{table}[h]\setlength\small
	\centering
	\begin{tabular}{|c|c|c|c|}
		\hline
		$k$ & $\E V_{T,x}^k$ & $\E V_T^k$ &\small $\E V_{T,x}^k/\E V_T^k$ \\
		\hline\hline
		$3$ & $1/375$ & $31/9\,000$ & $24/31 \approx 0.774194$ \\[1ex]
		$4$ & $13/21\,600$ & $1/900$ & $13/24 \approx 0.541667$ \\[1ex]
		$5$ & $151/987\,840$ & $1\,063/2\,469\,600$ & $755/2\,126 \approx 0.355127$ \\[1ex]
		$6$ & $1/23\,520$ & $403/2\,116\,800$ & $90/403 \approx 0.223325$ \\[1ex]
		$7$ & $83/6\,531\,840$ & $211/2\,268\,000$ & $2\,075/15\,192 \approx 0.136585$\\[1ex]
		$8$ & $73/18\,144\,000$ & $13/264\,600$ & $511/6\,240 \approx 0.081891$\\[1ex]
		$9$ & $1\,433/1073318400$ & $2\,593/93\,915\,360$ & $10\,031/207\,440 \approx 0.0483562$\\[1ex]
		$10$ & $647/1\,405\,071\,360$ & $697/42\,688\,800$ & $22\,645/802\,944 \approx 0.0282025$\\
		\hline
	\end{tabular}
	\caption{$\E V_{T,x}^k$ and $\E V_T^k$ for a triangle $T$ of volume $1$}\label{tab_triangle}
\end{table}
\\ Note that for $k=2$, we have 
$\E V_{T,x}^2/\E V_T^2 = 1 . $

\qed

\bigskip
\parindent=0pt

\bigskip
\begin{samepage}
Benjamin Reichenwallner\\
Institut f\"ur Mathematik \\
Universit\"at Salzburg\\
Hellbrunner Straße 34\\
5020 Salzburg, Austria\\
e-mail: benjamin.reichenwallner@sbg.ac.at
\end{samepage}
\bigskip

\begin{samepage}
Matthias Reitzner\\
Institut f\"ur Mathematik\\
Universit\"at  Osnabr\"uck\\
Albrechtstr. 28a \\
49076 Osnabr\"uck, Germany\\
e-mail: matthias.reitzner@uni-osnabrueck.de
\end{samepage}

\bigskip

\end{document}